\documentclass[reqno]{amsart}
\usepackage{graphicx}
\usepackage{amssymb}
\usepackage{cite}
\usepackage{amsmath}
\usepackage{latexsym}
\usepackage{amscd}
\usepackage{amsthm}
\usepackage{mathrsfs}
\usepackage{url}
\usepackage[linktocpage=true,colorlinks,citecolor=blue,linkcolor=blue,urlcolor=blue]{hyperref}
\usepackage{ textcomp}

\newtheorem{theorem}{Theorem}[section]

\newtheorem{lemma}{Lemma}[section]
\newtheorem{definition}{Definition}[section]

\newtheorem{corollary}{Corollary}[section]

\newtheorem{example}{Example}[section]
\theoremstyle{definition}

\theoremstyle{remark}
\newtheorem{remark}{Remark}[section]
\numberwithin{equation}{section}

\newfont{\bb}{msbm10 at 10pt}
\newfont{\bbt}{msbm10 at 7pt}

\def\r{\hbox{\bb R}}
\def\s{\hbox{\bb S}}

\begin{document}
\title[FREE BOUNDARY MINIMAL HYPERSURFACES IN THE
SCHWARZSCHILD SPACE]{ON free boundary MINIMAL HYPERSURFACES IN THE
RIEMANNIAN SCHWARZSCHILD SPACE}

\author[Barbosa]{Ezequiel Barbosa}     
\address{Departamento de Matem\'atica, Universidade Federal de  Minas Gerais,  Belo Horizonte-Brazil}
\email{ezequiel@mat.ufmg.br}

 \author[Espinar]{Jos\'e  M. Espinar}
 \address{Universidad de C\'adiz - Spain}
 \email{josemaria.espinar@uca.es}
 \thanks{The first author, Ezequiel Barbosa, is partially supported by Brazilian CNPq (Grant 312598/2018-1). The second author, Jos\'e M. Espinar, is partially supported by Spanish MEC-FEDER (Grant MTM2016-80313-P and Grant RyC-2016-19359) and Junta de Andaluc\'{i}a PAIDI (Grant P18-FR-4049).}
\maketitle


\begin{abstract}
In contrast with the 3-dimensional case (cf. \cite{RaMo}), where rotationally symmetric totally geodesic free boundary minimal surfaces have Morse index one; we prove in this work that the Morse index of a free boundary rotationally symmetric totally geodesic hypersurface of the $n$-dimensional Riemannnian Schwarzschild space with respect to variations that are tangential along the horizon is zero, for $n\geq4$.

Moreover, we show that there exist non-compact free boundary minimal hypersurfaces which are not totally geodesic, $n\geq 8$, with Morse index equal to $0$. Also, it is shown that, for $n\geq4$, there exist infinitely many non-compact free boundary minimal  hypersurfaces, which are not congruent to each other, with infinite Morse index.

We also study the density at infinity of a free boundary minimal hypersurface with respect to a minimal cone constructed over a minimal hypersurface of the unit Euclidean sphere. We obtain a lower bound for the density in terms of the area of the boundary of the hypersurface and the area of the minimal hypersurface in the unit sphere. This lower bound is optimal in the sense that only minimal cones achieve it.
\end{abstract}


\section{Introduction}

In General Relativity, asymptotically flat manifolds of non-negative scalar curvature play a crucial role, arising as isolated gravitational systems, due to the physical properties of space-times containing such manifolds as initial data sets (see \cite{C16,CCE16} and references therein). Another key ingredient to study physical properties of these manifolds are the minimal hypersurfaces contained on them; as the Positive Mass Theorem, by R. Schoen and Yau \cite{RSchSYau81,RSchSYau17}, has shown. 

The most simple solution to the Einstein equations that contains an asymptotically flat manifold of non-negative scalar curvature as an initial data set is, clearly, the Lorentz-Minkowski space being the Euclidean space the initial data set. The second one, found in 1916 by K. Schwarzchild \cite{Sch16}, is the Schwarzchild vacuum which is a solution to the Einstein field equations that describes the gravitational field outside a spherical mass; which initial data set is known as the Riemannian Schwarschild manifold. The important difference of this particular solution is the existence of a "singularity"; that we can think of as the event horizon of a "Black Hole". 

The aim of this work is to study stability properties (with respect to the second variation of the area) of minimal hypersurfaces with boundary meeting orthogonally the horizon of the Riemannian Schwarzchild space. We consider, for each $n\geq 3$ and $m>0$, the $n$-dimensional domain
\[
M^{n}:=\left\{x\in\mathbb{R}^{n}\,;\,\, |x|\geq R_{0} \right\}\,,
\]
being $R_{0} := \left(\frac{m}{2}\right)^{\frac{1}{n-2}}$, endowed with the Riemannian metric
\[
g_{Sch}=\left( 1+\frac{m}{2|x|^{n-2}} \right)^{\frac{4}{n-2}}\,\delta\,,
\]
where $|x|^2=\sum\limits_{i=1}^{n}x_i^2$ and $\delta$ denotes the Euclidian metric. We can check that the boundary of $M$, the horizon $S_{0}=\{x\in\mathbb{R}^{n}\,;\,\, |x|=R_{0} \}$, is a closed totally geodesic hypersurface in $M$. Also, we consider the totally geodesic Euclidean hypersurfaces which are the rotations of the coordinate hyperplane $\Sigma _0$ through the origin minus the open ball $B_{0}:=B( R_{0})$ of radius $R_0$ centered at the origin: 
\begin{equation}\label{Sigma}
\Sigma_0=\{x\in\mathbb{R}^{n}\,;\,\, x_{n}= 0 \,,\,\,|x|\geq R_{0} \}.
\end{equation}

This hypersurface is a properly embedded free boundary totally geodesic, in particular minimal, hypersurface in $M$, i.e., the boundary $\partial \Sigma_0$ coincides with $\Sigma_0 \cap \partial M$, its mean-curvature with respect to $g_{Sch}$ vanishes and the boundary $\partial \Sigma_0$ meets $\partial M$ orthogonally.
 
In the 3-dimensional Riemannnian Schwarzschild space, R. Montezuma \cite{RaMo} computed the Morse index of $\Sigma_0$, i.e., the maximum number of directions, tangential along $\partial M$, in which the surface can be deformed in such a way that its area decreases. It was proved that the Morse index of $\Sigma_0$ is one. One of the purposes of this paper is to compute the Morse index of $\Sigma_0$ in the $n$-dimensional Riemannnian Schwarzschild space, $n\geq4$. We obtain the following result:

\begin{theorem}\label{main1}
The Morse index of $\Sigma_0 \subset (M^{n},g_{Sch})$ given by \eqref{Sigma}, up to a space rotation, is zero. 
\end{theorem}

A question that arises here is the following one: Is the hypersurface $\Sigma_0$ the only one, among properly embedded free boundary minimal hypersurfaces, with Morse index equal to zero, up to space rotation, in the Schwarszchild space when $n\geq4$? We answer this question at least for dimension $n\geq 8$.

\begin{theorem}\label{main1a}
There are properly embedded free boundary minimal hypersurfaces, not totally geodesic, in the $n$-dimensional Riemannnian Schwarzschild space with Morse index equal to 0 when $n \geq 8$.
\end{theorem}

We even show a more general result that we describe now. Let 
$$C_{\Gamma}:=\{ \lambda y \, ; \,y \in \Gamma^{n-2}, \,\lambda\in (0,\infty)\}$$be a cone in $\mathbb{R}^{n}$, $n\geq 4$, with vertex at the origin, here $\Gamma^{n-2}$ is an embedded closed orientable minimal hypersurface in $\mathbb{S}^{n-1}$. We refer to $C_{\Gamma}$ as the {\it minimal cone over ${\Gamma}$}. Finally, consider $\Sigma_{\Gamma}=\{x\in C_{\Gamma}\,;\,\, \,\,|x|\geq R_{0} \}$. This  hypersurface is a properly embedded free boundary minimal hypersurface in $M^n$. Note that $\Sigma_0=\Sigma_{\Gamma_0}$ when $\Gamma_0$ is an equator in $\mathbb{S}^{n-1}$. Hence, we show:

\begin{theorem}\label{main2}
The Morse index of the hypersurface $\Sigma_{\Gamma} \subset (M^{n},g_{Sch})$, $4\leq n\leq7$, up to a space rotation, is finite if, and only if, $\Gamma$ is totally geodesic in $\mathbb{S}^{n-1}$ (i.e., $\Sigma_{\Gamma}=\Sigma_0$); in particular, the Morse index is zero.
\end{theorem}

\begin{remark}
Using the techniques of Theorems \ref{main1a} and \ref{main2}, one is able to obtain an upper bound for the maximal annular domain of stability of $\Sigma _\Gamma$ as in \cite[Theorem 1.2]{RaMo}. However, in our case, due to the lack of an explicit solution we are not able to compute the sharp one. 
\end{remark}

As we can see in Theorem \ref{main2}, all non-totally geodesic cones have infinite Morse index when $4\leq n\leq7$. However, for every dimension $n\geq8$, we find non-totally geodesic cones with finite index (in fact, these cones are stable), and infinitely many other cones with infinite Morse index.  This is a consequence of a link between the stability of a general minimal cone $\Sigma_\Gamma$ and the first eigenvalue of the Jacobi operator of $\Gamma \subset \mathbb{S}^{n-1}$. 

\begin{theorem}\label{newMAIN}
Let $\Sigma_{\Gamma}$ a cone in the $n$-dimensional Riemannnian Schwarzschild space $M^n$, $n\geq4$. If 
\[
4\lambda_1(\Gamma)+(n-2)(n-4)\geq0 ,
\]
where $\lambda_1(\Gamma)$ is the first eigenvalue of the Jacobi operator of $\Gamma$ in $\mathbb{S}^{n-1}$, then $\Sigma_{\Gamma}$ is stable.
\end{theorem}

\begin{theorem}\label{main3}
For every $n\geq4$, there exist infinitely many non-compact free boundary minimal  hypersurfaces in the $n$-dimensional Riemannnian Schwarzschild space, which are not congruent to each other, with infinite Morse index.
\end{theorem}

\begin{remark}
It would be interesting to study the existence or not of properly embedded free boundary minimal hypersurfaces in the Riemannnian Schwarzschild space, $n\geq 4$, with Morse index one. A good candidate might be to construct rotationally symmetric catenoidal type minimal hypersurfaces. 
\end{remark}

Finally, we consider the density at infinity of a non-compact free boundary minimal  hypersurface $\Sigma$ with respect to a minimal cone $\Sigma_{\Gamma}$.  We obtain an inequality which relates the area of the boundary $\partial\Sigma$ of $\Sigma$ and the
density at infinity $\Theta_{\Gamma}(\Sigma)$. In fact, we show that the ratio between the area of the boundary of $\partial\Sigma$ and the area of $\Gamma$ is a lower bound for the density at infinity. This is the content of the next result.

\begin{theorem}\label{Density}
The area of the boundary of $\Sigma$ satisfies
\begin{equation}\label{CorEqualityA}
area(\partial \Sigma)\leq 2m|\Gamma|\Theta_{\Gamma}(\Sigma)\,.
\end{equation}
Moreover, equality holds if and only if $\Sigma$ is a minimal cone; in such case
$$ |\Gamma| \Theta_{\Gamma}(\Sigma)  = |R_0 ^{-1} \partial \Sigma|,$$ 
where $R_0 ^{-1}\partial\Sigma \subset \mathbb{S}^{n-1}$ is nothing but the dilation of $\partial \Sigma$ into the $(n-1)-$sphere of radius one.
\end{theorem}

For the three-dimensional case, there is only one free boundary minimal cone, namely $\Sigma_0$. The above theorem, in the case $n=3$, was obtained by R. Montezuma \cite{RaMo} with the rigidity part being exactly $\Sigma_0$. In our case, for dimension $n\geq4$, we might consider different minimal cones and, hence, different densities at infinity with respect to those cones. Therefore, the rigidity part depends on the minimal hypersurface $\Gamma$ used in the construction of the cone and the density which is being considered. However, we can use the recent resolution of the Willmore conjecture due to F. Marques and A. Neves \cite{MaNe} to obtain a more specific rigidity result when $n=4$:
 
\begin{theorem}\label{ThMN}
Let $\Sigma$ be a properly embedded free boundary minimal hypersurface in $(M^4, g_{Sch})$. Assume that there exists an embedded minimal hypersurface $\Gamma \subset \mathbb{S}^{3}$ so that achieves the equality in \eqref{CorEqualityA} and $ |\Gamma| \Theta _\Gamma (\Sigma ) \leq 2\pi ^2 $. Then, up to a rotation, either $\Sigma \equiv \Sigma_0$ or $\Sigma$ is a minimal cone over a Clifford torus.
\end{theorem} 

In any dimension, we can also obtain a more specific rigidity result using Allard's Regularity Theorem \cite{Allard}.
 
\begin{theorem}\label{ThAllard}
Let $\Sigma$ be a properly embedded free boundary minimal hypersurface in $(M^n, g_{Sch})$. There exists a constant $\epsilon (n) >0$, depending on the dimension, so that if there exists an embedded minimal hypersurface $\Gamma \subset \mathbb{S}^{n-1}$ that achieves the equality in \eqref{CorEqualityA} and $  |\Gamma|\Theta _\Gamma (\Sigma ) < \omega _{n-2} + \epsilon (n)$, where $\omega _{n-2}$ is the volume of $\mathbb{S}^{n-2}$; then, up to a rotation, $\Sigma \equiv \Sigma_0$.
\end{theorem}


\section{Preliminaries}

Let $\Sigma$ be a properly embedded hypersurface in $(M^n, g_{Sch})$. Consider the class of $C^1$ vector fields $X$ on $M^n$ that are tangential along $\partial M$ and compactly supported. We use $\{\varphi(t,\cdot)\}$ to denote the one-parameter family of diffeomorphisms associated to $X$, and use it to obtain a variation of $\Sigma$ with variational vector field $X$, i.e. we consider $\varphi(t, \Sigma)=\{\varphi(t, x);\, x\in\Sigma\}$. The first derivative of the area functional in the direction of $X$ is given as
\[
\frac{d}{dt}\big{|}_{t=0}area_{g_{Sch}}(\varphi(t,\Sigma))=\int_{\Sigma}g_{Sch}(X,H)dv_{\Sigma}+\int_{\partial\Sigma}g_{Sch}(X,\nu)ds ,
\]
where $H$ and $\nu$ denote the mean curvature and outward pointing unit co-normal vectors of $\partial\Sigma$, respectively. We also denote by $g$ the induced metric on $\Sigma$. It follows from this formula that free boundary minimal hypersurfaces are precisely the critical points of the area functional with respect to tangential variations.

Let $\xi$ denote a globally defined unit normal vector field along $\Sigma$, which exists since $\Sigma$ is properly embedded. From now on, we restrict our attention to smooth variational vector fields that are normal to $\Sigma$, i.e., along the hypersurface $X=u \xi$ for some smooth function $u: \Sigma\rightarrow  \mathbb{R}$. The free boundary condition implies that $X$ is an admissible tangential variational vector field. Assuming that $\Sigma$ is properly embedded and minimal ($H=0$), the second derivative of the area functional can be computed as
\[
\frac{d^2}{dt^2}\big{|}_{t=0}area_g(\varphi(t,\Sigma))=Q_{\Sigma}(u,u)\,,
\]
where $Q_{\Sigma}(\cdot, \cdot)$ is the quadratic form given by
\begin{equation}\label{form1}
\begin{split}
Q_{\Sigma}(u,u) & = - \int_{\Sigma}u\left( \Delta_{\Sigma}u + (\overline{{\rm Ric}}(\xi ,\xi) +|A_{\Sigma}|^2)u \right)dv_{\Sigma} \\ 
 & \qquad \qquad +\int_{\partial \Sigma}u\left(\frac{\partial u}{\partial \nu} - A_{\partial M}(\xi ,\xi )u \right)ds\,,
\end{split}
\end{equation}where $\overline{{\rm Ric}}$ is the Ricci  curvature of $(M^{n}, g_{Sch})$ and we are using the notations $A_{\Sigma}$ and $A_{\partial M}$ for the second fundamental forms of $\Sigma$ and $\partial M$, respectively:
\[
A_{\partial M}(V,W) = -g_{Sch}(\bar{\nabla} _{\eta}V,W)
\] 
where $\eta$ is the inwards pointing unit normal along $\partial M$, and 
\[
A_{\Sigma}(V,W) = -g_{Sch}(\bar{\nabla}_\xi V,W)\,.
\]

Fix $R > R_{0}$, as in \cite{RaMo}, we make the following definition:

\begin{definition}
Let $\Sigma(R) =\Sigma \cap \{|x| \leq R\}$. The Morse index for functions  vanishing on $\{|x| = R\}$ , ${\rm Ind}_{F}(\Sigma(R))$, of $\Sigma(R)$ is defined as the maximal dimension of a linear subspace $V$ of smooth functions $u:\Sigma(R)\rightarrow \mathbb{R}$ vanishing on $\{|x| = R\}$ such that $Q_{\Sigma}(u, u)<0$, for all $u \in V\setminus \{0\}$. 

The Morse index, ${\rm Ind}(\Sigma)$, of $\Sigma$ is defined as  
$${\rm Ind}(\Sigma) := \limsup\limits_{R\rightarrow +\infty}{\rm Ind}_{F}(\Sigma(R)),$$possibly being infinite. Moreover, when ${\rm Ind}(\Sigma) = 0$, we say that $\Sigma$ is stable.
\end{definition}

Equivalently, the Morse index ${\rm Ind}_F(\Sigma(R))$ is the number of negative eigenvalues, counting multiplicities, of the problem (cf. \cite[Definition 2.1]{RaMo})
\[
(F)\,\,\,\left\{ \begin{array}{ccc}
      J_{\Sigma}\psi =-\beta \psi & \text{ in } & \Sigma(R) \\
     \psi=0 &\text{ on } &S(R)\cap \Sigma \\
     \dfrac{\partial \psi}{\partial \nu}=0  & \text{ on } & S_{0}\cap\partial\Sigma
    \end{array} \right.  
\]
where $J_{\Sigma}\psi := \Delta_{\Sigma} \psi + (\overline{{\rm Ric}}(\xi ,\xi) +|A_{\Sigma}|^2)\psi$ and $S(R)=\mathbb{S}^{n-1}(R)$ is the sphere centered at the origin and radius $R$. Henceforth, we will consider $n\geq 4$. The case $n=3$ was considered in \cite{RaMo}.

\subsection{Fischer-Colbrie Criterion} 

A first task to do is to characterize the stability, ${\rm Ind}(\Sigma) =0$, in terms of subsolutions of the differential equation $J_{\Sigma} u \leq 0$, this is known as the Fischer-Colbrie Criterion and its proof follows from the original one given in (cf. \cite{FC85}); however we include it here for the sake of completeness. 

\begin{lemma}[Fischer-Colbrie Criterion]\label{LemFC}
Let $\Sigma$ be a properly embedded free boundary minimal hypersurface in the $n$-dimensional Schwarszchild Riemannian manifold $(M^n, g_{Sch})$, $n\geq3$. Let  $J_{\Sigma}$ the Jacobi operator of $\Sigma$. 

If there is a smooth positive function $u$ on $\Sigma$ such that
\[
(\ast) \quad \left\{ \begin{array}{ccc}
     J_{\Sigma}u \leq0& \text{ in } & \Sigma ,\\
     \dfrac{\partial u}{\partial \nu}=0  & \text{ on } & S_{0}\cap\partial\Sigma ,
    \end{array} \right.  
\]then $\Sigma$ is stable, i.e. ${\rm Ind}(\Sigma)=0$.
\end{lemma}
\begin{proof}
Assume there exists a smooth positive function $u$ in $\Sigma$ satisfying ($\ast$). Given $R>R_0$, take the first eigenvalue $\lambda_1(R)$ and the first eigenfunction $f$ on $\Sigma(R)$ associated with $\lambda_1 (R)$:
\[
\left\{ \begin{array}{ccc}
J_{\Sigma}f=-\lambda_1 (R) f & \text{ in } & \Sigma(R) \\
     f=0 &\text{ on } &S(R)\cap \Sigma \\
     \dfrac{\partial f}{\partial \nu}=0  & \text{ on } & S_{0}\cap\partial\Sigma\,.
    \end{array} \right.  
\]
Set $\psi=\dfrac{f}{u}$. Note that $u\psi=0$ on $S(R)\cap\partial\Sigma$ and $\dfrac{\partial u\psi}{\partial \nu}=0$ on $S_0\cap\partial\Sigma$. Hence, integrating by parts, we obtain that 
\begin{eqnarray*}
\lambda_1(R)\int_{\Sigma(R)}f^{2}dv&=&\int_{\Sigma(R)}|\nabla_{\Sigma}f|^2-f^{2}(\overline{{\rm Ric}}(\xi ,\xi) +|A_{\Sigma}|^2)dv_{\Sigma}\\
&=&\int_{\Sigma(R)}|\nabla_{\Sigma} (u\psi)|^2-(u\psi)^2(\overline{{\rm Ric}}(\xi ,\xi) +|A_{\Sigma}|^2)dv\\
&=&\int_{\Sigma(R)}-u\psi \Delta_{\Sigma}(u\psi)-(\overline{{\rm Ric}}(\xi ,\xi) +|A_{\Sigma}|^2)u^2\psi^2dv+\int_{\partial \Sigma(R)}u\psi\frac{\partial u\psi}{\partial \nu}ds\\
&=& \int_{\Sigma(R)}u\psi^2(-\Delta_{\Sigma}u -u(\overline{{\rm Ric}}(\xi ,\xi) +|A_{\Sigma}|^2)) -2u\psi \, g(\nabla _\Sigma u,\nabla_\Sigma \psi)-u^2\psi\Delta_{\Sigma} \psi dv\\
&\geq&\int_{\Sigma(R)} -2u\psi \, g(\nabla_\Sigma u,\nabla_\Sigma\psi)-u^2\psi\Delta_{\Sigma} \psi dv\\
&=&\int_{\Sigma(R)}|\nabla_{\Sigma}\psi|^2u^2-div_{\Sigma}(u^2\psi\nabla_{\Sigma}\psi)dv\\
&=&\int_{\Sigma(R)}|\nabla_{\Sigma}\psi|^2u^2dv\geq0\,.
\end{eqnarray*}
Hence, $\lambda_1(R)\geq0 $. This is enough to obtain that $Ind_F(\Sigma(R))=0$, and, consequently, $Ind(\Sigma)=0$.
\end{proof}

\begin{remark}
We will use the Fischer-Colbrie Criterion in order to prove Theorems \ref{main1} and \ref{main1a}.
\end{remark}


\subsection{Cones in the Schwarschild space}

Let us consider now a class of minimal hypersurfaces in the Schwarschild space. Let 
$$C_{\Gamma}:=\{ \lambda y \, ; \,y \in \Gamma^{n-2}, \,\lambda\in (0,\infty)\}$$
be a cone in $\mathbb{R}^{n}$, $n\geq 4$, with vertex at the origin. Here $\Gamma^{n-2}$ is an embedded closed orientable minimal hypersurface in $\mathbb{S}^{n-1}$, the standard $(n-1)-$dimensional sphere centered at the origin. We refer to $C_{\Gamma}$ as the {\it minimal cone over ${\Gamma}$}. Note that $\Sigma_0=\Sigma_{\Gamma_0}$, where $\Gamma_0$ is an equator in $\mathbb{S}^{n-1}$. 

Observe that the support function $\rho=\left\langle x,\xi \right\rangle$, $x\in C_{\Gamma}$, satisfies $\rho \equiv 0$ in $C_{\Gamma}$. Hence, $\Sigma_{\Gamma}=\{x\in C_{\Gamma}\,;\,\, \,\,|x|\geq  R_{0} \}$ is a properly embedded free boundary minimal hypersurface in $(M^{n},g_{Sch})$ if, and only if, it is a properly embedded free boundary minimal hypersurface in $(\r^{n}\setminus \overline{B_{0}} , \delta)$, where $\delta$ the Euclidean metric.

\begin{example}[Clifford cones]
As an example, consider the Clifford torus 
$$\mathbb{T}_{m,n-2}:=\mathbb{S}^{m}(\lambda_1) \times \mathbb{S}^{(n-2)-m}(\lambda_2),$$ where $\lambda_1=\sqrt{\frac{m}{n-2}}$, $\lambda_2 = \sqrt{\frac{(n-2)-m}{n-2}}$ and $1 \leq m \leq n-2$. Since $\mathbb{T}_{m,n}$ is a minimal  hypersurface in $\mathbb{S}^{n-1}$, the cone $C_{m,n}=\{ \lambda y \ ; y \in \mathbb{T}_{m,n-2}, \lambda\in (R_0,\infty)\}$ is a minimal hypersurface in the Riemannian Schwarzschild space $(M^n, g_{Sch})$. Here, $\s^{p}(\lambda)$ denotes the $p-$dimensional sphere of Euclidean radius $\lambda$ centered at the origin.
\end{example}


\subsubsection{Relating the Schwarzschild and Euclidean geometries of cones}

Let $\Sigma : = \Sigma _{\Gamma}$ be a cone as defined above. The Schwarzschild metric is conformal to the Euclidan metric by $g_{Sch}=f^2\delta$, where
\begin{equation}\label{Factorf}
f( |x| )=\left( 1+\frac{m}{2|x|^{n-2}} \right)^{\frac{2}{n-2}}\, , 
\end{equation}then, if we denote by $g$ and $g_\delta$ the induced metric on $\Sigma$ in the Schwarzschild and Euclidean metric respectively, we can observe that both metrics are conformal and related by 
\begin{equation}\label{FactorF} 
g :=F^{\frac{4}{N-2}}g_\delta , \text{ where } F = \left( 1+\frac{m}{2|x|^{n-2}} \right)^{\frac{N-2}{n-2}} \text{ and } N= n-1.
\end{equation}

Also, we can relate the second fundamental forms using \cite[Lemma 10.1.1]{RaLo} and that $\Sigma $ is a cone, specifically
\begin{equation}\label{Second}
|A_{(\Sigma, g_{Sch})}|^2 =  |A_{(\Sigma,\delta)}|^2 F^{-\frac{4}{N-2}},
\end{equation}where $A_{(\Sigma,g_{Sch})}$ and $A_{(\Sigma, \delta)}$ are the second fundamental forms of $\Sigma$ as a hypersurface in the Schwarschild and Euclidean metric respectively. Recall that, since $\Sigma$ is minimal, $|A_{(\Sigma,\delta)}|^2 = - {\rm S}_{(\Sigma , \delta)}$ is nothing but the scalar curvature of $\Sigma$ as a hypersurface in the Euclidean space $(\r ^{n} \setminus \overline{B_{0}} , \delta)$. 

Finally, using the Yamabe equation (cf. \cite[Section 1]{Escobar}) for the conformal metrics \eqref{FactorF} and $ |A_{(\Sigma, g_{Sch})}|^2 = - {\rm S}_{(\Sigma , \delta)} F^{-\frac{4}{N-2}}$, we obtain
\begin{equation}\label{LaplF}
-F^{-\frac{N+2}{N-2}}\Delta_{\delta}F =\frac{N-2}{4(N-1)}\left( {\rm S}_{\Sigma} + |A_{(\Sigma, g_{Sch})}|^2\right),
\end{equation}where $\Delta _\delta$ denotes the Laplacian with respect to the metric $g_\delta$


\section{Bounds on the Morse index}

Since round spheres $S(R)\equiv \s ^{n-1}(R)$, $R\geq R_{0}$, are totally umbilic in the Euclidean space and the Schwarzschild metric is conformal to the Euclidean metric, it follows that $S(R)$ is totally umbilic in the Schwarszchild space. In particular (cf. \cite[Lemma 10.1.1]{RaLo}), one can easily see that the second fundamental form of $S(R)$ with respect to the outer unit normal in the Schwarschild space is given by
\[
A_{S(R)}(\nu,\nu)= - \kappa(R) \, g_{Sch}(\nu,\nu) 
\]for every $\nu \in T_{x}S(R)$, $|x| =R$ and $R\geq R_{0}$, where 
\[
\kappa(R) =  \frac{(R^{n-2}-R^{n-2}_{0}) R}{(R^{n-2}+R^{n-2}_{0})^{\frac{n}{n-2}}} .
\]

For each $R>R_0$, consider the compact domain $\Omega(R)=\overline{B(R)} \setminus B_0$. Since a properly embedded free boundary minimal hypersurface $\Sigma$ in $(M^n,g_{Sch})$ has boundary $\partial\Sigma\subset S_0$, then $\Sigma(R)$, the connected component of $\Sigma \cap \Omega (R)$ whose boundary contains $\partial \Sigma$, is a properly embedded minimal hypersurface in $\Omega(R)$ whose boundary satisfies $\partial \Sigma(R)\subset S_0\cup S(R)$. Note also that along the components of the boundary $\partial \Sigma(R)$ in $S_0$, $\Sigma(R)$ meets $S_0$ orthogonally. However, the components of the boundary $\partial\Sigma(R)$ in $S(R)$ might fail to satisfy this orthogonality condition. This means that it can happen that $\Sigma(R)$ is not a properly embedded free boundary minimal hypersurface in $\Omega(R)$. 

When $\Sigma(R)$ is a properly embedded free boundary minimal hypersurface in $\Omega(R)$, i.e., along all the components of the boundary $\partial \Sigma(R)$, $\Sigma(R)$ meets the boundary $\partial\Omega(R)$ orthogonally (this is the case when $\Sigma$ is a cone), we can consider the Morse index (quadratic) form, $Q_{\Sigma}(R)$, of $\Sigma(R)$ given by
\begin{equation}\label{RQuadratic}
\begin{split}
Q_{\Sigma}(R)(\psi,\varphi) & = - \int_{\Sigma(R)}\psi\left( \Delta\varphi + (\overline{{\rm Ric}}(\xi ,\xi) +|A_{\Sigma}|^2)\varphi \right)dv_{\Sigma} \\ 
 & \qquad \qquad +\int_{\partial \Sigma(R)}\psi\left(\frac{\partial \varphi}{\partial \nu} - q\varphi \right)ds\,,
\end{split}
\end{equation}
where 
\[
q=\,\,\,\, \left\{ \begin{array}{ccc}
     0      & \text{ on } & S_0 ,\\
     \kappa(R) & \text{ on } & S(R) .
    \end{array} \right.    
\]

It is worth to mention here that, for every $R>R_0$, we always have the quadratic form $Q_{\Sigma}$ related with the index ${\rm Ind}_{F}(\Sigma(R))$ (see Definition \ref{form1}). If $\Sigma(R)$ is also a free boundary minimal hypersurface in $\Omega(R)$, we also have the quadratic form $Q_{\Sigma}(R)$ (given by \eqref{RQuadratic}). In this case, let us denote by ${\rm Ind}_{M}(\Sigma(R))$ the Morse index of $\Sigma(R)$ as a free boundary minimal hypersurface with respect to the quadratic form $Q_{\Sigma}(R)$. Observe that:

\begin{quote}
{\bf Claim A:} {\it  ${\rm Ind}_F(\Sigma(R))\leq {\rm Ind}_{M}(\Sigma(R))$.}
\end{quote} 
\begin{proof}[Proof of Claim A]
In fact, if $\psi\in V$, where $V$ is the space spanned by the eigenfunction $\psi_i$ of $(F)$, with $\beta_i<0$, we obtain  $Q_{\Sigma}(R)(\psi,\psi)=Q_{\Sigma}(\psi,\psi)<0 $; this proves Claim A.
\end{proof}

Moreover, since $q\geq 0$ and is non-identically zero, it follows (cf. \cite[Theorem 4.1]{TranZhou}) that the Morse index of $\Sigma(R)$, as a free boundary minimal hypersurface in $\Omega(R)$, is given by the addition ${\rm Ind}_M(\Sigma(R))= {\rm Ind}_{D}(\Sigma(R))+ {\rm Null}_D (\Sigma(R))+{\rm Ind}_{R}(\Sigma(R))$, where  ${\rm Ind}_{D}(\Sigma(R))$ is the number of non-positive eigenvalues, counting multiplicity, of the problem
\[
(D)\,\,\, \left\{ \begin{array}{ccc}
     \Delta_{\Sigma} v + (\overline{{\rm Ric}}(\xi ,\xi) +|A_{\Sigma}|^2)v =-\delta v & \text{ in } & \Sigma(R) ,\\ 
     v=0 & \text{ on } & \partial \Sigma(R),
    \end{array} \right.  
\]
${\rm Null}_D(\Sigma(R))$ is the nullity of the above problem, and ${\rm Ind}_{R}(\Sigma(R))$ is the number of eigenvalues smaller than $1$, counting multiplicity, of the problem 
\[
(R)\,\,\, \left\{ \begin{array}{ccc}
     \Delta_{\Sigma} u + (\overline{{\rm Ric}}(\xi ,\xi) +|A_{\Sigma}|^2)u =0 & \text{ in } & \Sigma(R),\\
     \dfrac{\partial u}{\partial \nu}=\lambda q u & \text{ on } & \partial \Sigma(R).
    \end{array} \right.  
\]

Reasoning as above, we can also easily show:
\begin{quote}
{\bf Claim B:} {\it  ${\rm Ind}_D(\Sigma(R))\leq {\rm Ind}_{F}(\Sigma(R))$.}
\end{quote} 


\subsection{Relating the Euclidean and Schwarschild index forms}

From now on, $\Sigma$ will always denote a minimal cone $\Sigma := C_{\Gamma} \cap \{ |x| \geq R_{0} \}$. As we have pointed out above, a minimal cone, $\Sigma$, in the Schwarshild metric is also minimal in the Euclidean metric, and viceversa. Hence, we will relate the Schwarschild index form $Q_{\Sigma}(R)$ and the Euclidean index form $Q_{\delta}(R)$ for the Dirichlet problem; i.e., 
\begin{equation}\label{DeltaQuadratic}
Q_{\delta}(R)(\psi,\varphi)  = - \int_{\Sigma \cap \Omega (R)}\psi\left( \Delta_{\delta}\varphi +|A_{\Sigma}|^2)\varphi \right)d\delta , \, \, \psi = 0 \text{ on } \partial \Sigma (R) ;
\end{equation}where $d\delta$ is the area element associated to $g_{\delta}$.

\begin{lemma}\label{LemRelation}
If $\psi$ is zero on $\partial \Sigma(R)$, $R>R_{0}$, then 
\begin{equation}\label{Relation}
Q_\Sigma (R)(F^{-1}\psi, F^{-1}\psi) = Q_{\delta}(R)(\psi,\psi) -\frac{N}{N-2}\int_{\Sigma(R)} \left( F^{-1}\Delta_{\delta} F\right)\psi^2d\delta.
\end{equation}
\end{lemma}
\begin{proof}
First, it follows from the Gauss equation that
\begin{equation}\label{EqGauss}
\overline{{\rm Ric}}(\xi,\xi)+|A_{(\Sigma,g)}|^2 =- \frac{{\rm S}_{\Sigma}}{2}+\frac{|A_{(\Sigma,g)}|^2}{2}
\end{equation}where ${\rm S}_{\Sigma} $ is the scalar curvature of $(\Sigma ,g)$. Second, denote by $L_g$ the conformal Yamabe operator of the metric $g$ on $\Sigma$ given by
\begin{equation}\label{EqYamabe}
L_{g}u=\Delta_{\Sigma}u-\frac{N-2}{4(N-1)}{\rm S}_{\Sigma}u\, ,
\end{equation}

Hence, for any smooth function $u \in C^{\infty}(\Sigma)$, we obtain
\begin{equation*}
\begin{split}
\Delta_{\Sigma}u+(\overline{{\rm Ric}}(\xi,\xi)+|A_{(\Sigma,g)}|^2)u &=^{\eqref{EqYamabe}} L_g(u)+\frac{N-2}{4(N-1)}{\rm S}_{\Sigma}u+(\overline{{\rm Ric}}(\xi,\xi)+|A_{(\Sigma,g)}|^2)u\\
&=^{\eqref{EqGauss}}L_g(u)-\frac{N}{4(N-1)}{\rm S}_{\Sigma} u + \frac{|A_{(\Sigma,g)}|^2}{2}u\\
&=^{\eqref{LaplF}} L_{g}u+\frac{N}{N-2}\left(F^{-\frac{N+2}{N-2}}\Delta_{\delta}F\right)u \\ 
 & \qquad + \left( 1-\frac{N-2}{4(N-1)} \right)|A_{(\Sigma,\delta)}|^2F^{-\frac{4}{N-2}}u ,
\end{split}
\end{equation*}where we have used \eqref{Second}. On the other hand, since $g=F^{\frac{4}{N-2}}g_\delta$, it follows (cf. \cite[Section 1]{Escobar}) that 
\begin{equation*}
L_{g}(F^{-1}\psi)=F^{-\frac{N+2}{N-2}}L_{\delta}(\psi)=F^{-\frac{N+2}{N-2}}\left(\Delta_{\delta}\psi-\frac{N-2}{4(N-1)}{\rm S}_{(\Sigma,\delta)}\psi\right)\, .
\end{equation*}

Therefore, using the above two equations and \eqref{Second}, the Jacobi operator $J_{\Sigma} v=\Delta_{\Sigma}v+(\overline{{\rm Ric}}(\xi,\xi)+|A_{(\Sigma,g)}|^2)v$ on $\Sigma$ satisfies
\begin{equation*}
\begin{split}
J_{\Sigma}(F^{-1}v)&=L_{g}(F^{-1}v)+\frac{N}{N-2}\left(F^{-\frac{2N}{N-2}}\Delta_{\delta}F\right)v+\left( 1-\frac{N-2}{4(N-1)} \right)|A_{(\Sigma,\delta)}|^2F^{-\frac{N+2}{N-2}}v\\
&= F^{-\frac{N+2}{N-2}}\left(\Delta_{\delta}v-\frac{N-2}{4(N-1)}{\rm S}_{(\Sigma ,\delta)}v\right)+\frac{N}{N-2}\left( F^{-\frac{2N}{N-2}}\Delta_{\delta} F\right)v \\
& \qquad +\left( 1-\frac{N-2}{4(N-1)} \right)|A_{(\Sigma,\delta)}|^2F^{-\frac{N+2}{N-2}}v \\
&=F^{-\frac{N+2}{N-2}}\left( \left(\Delta_{\delta}v+ |A_{(\Sigma,\delta)}|^2v\right)+\frac{N}{N-2}\left( F^{-1}\Delta_{\delta} F\right)v \right)\,.
\end{split}
\end{equation*}

Consider the operator 
\begin{equation}\label{J}
\mathbb{J}_{\delta}v=\left( \Delta_{\delta}+|A_{(\Sigma,\delta)}|^2\right)v + \frac{N}{N-2}\left( F^{-1}\Delta_{\delta} F\right)v ,
\end{equation}hence we can re-write the above Jacobi operator as
$$ v \mathbb{J}_{\delta}(v)= (F^{-1} v)J_{\Sigma}(F^{-1}v) F^{\frac{2N}{N-2}} .$$

Thus, since the volume elements of $g$ and $g_{\delta}$ are related by $dv_{\Sigma} = F^{\frac{2N}{N-2}} d\delta$ from \eqref{FactorF}, we have obtained that the index forms $Q_{\Sigma}(R)$ and $Q_{\delta}(R)$, given by \eqref{RQuadratic} and \eqref{DeltaQuadratic} respectively, satisfy \eqref{Relation} for every $\psi$ that vanishes on $\partial\Sigma (R)$ as claimed.
\end{proof}

Finally, we must control the last term in the above equation \eqref{Relation}. Hence,
\begin{equation*}
\frac{N}{N-2}F^{-1}\Delta_{\delta} F=\frac{N}{N-2}F^{-1}(\frac{d^2F}{dr^2}+\frac{N-1}{r}\frac{d F}{dr})=(n-1)m\frac{2r^{n-4}}{\left(m+2r^{n-2}\right)^{2}}, 
\end{equation*}
for all $x \in \Sigma$ such that $|x|=r \geq R_{0}$. Thus, Lemma \ref{LemRelation} and the above observation implies

\begin{lemma}\label{LemRelation2}
If $\psi$ is zero on $\partial \Sigma(R)$, $R>R_{0}$, then 
\begin{equation*}
Q_{\Sigma}(R)(F^{-1}\psi,F^{-1}\psi)= Q_{\delta}(R)(\psi,\psi)-(n-1)m\int_{\Sigma(R)} \frac{2r^{n-4}}{\left(m+2r^{n-2}\right)^{2}}\psi^2d\delta\,.
\end{equation*}
\end{lemma}


\subsection{Euclidean index form over cones}

In this part we follow the seminal work of J. Simons \cite{Simons}. On the one hand, fix $R_{0} <R$ and consider the following initial value problem on the interval $[R_{0},R]$:
$$
\text{(IVP)} \qquad \left\{ \begin{matrix} 
-r^2\dfrac{d^2g}{dr^2}-(n-2)r\dfrac{dg}{dr} = \beta \, g & \text{ in } (R_{0} ,R) , \\[3mm]
g(R_{0}) = 0 = g(R) .& 
\end{matrix}\right.
$$

From \cite[Lemma 6.1.5]{Simons}, for each $j\in \mathbb{N}$, the function 
\[
g_j(r)= c_{j}r^{-\frac{(n-3)}{2}}\sin \left(\frac{j\pi}{\log (R/R_{0})}\log(r/R_{0}) \right), \text{ where } c_{j}^{-2}= \frac{\log (R/R_{0})}{2 j} , 
\]solves (IVP); that is, 
\[
-r^2\frac{d^2g_{j}}{dr^2}-(n-2)r\frac{dg_{j}}{dr}=\beta_j \, g_{j}\,, \text{ where }
\beta_j=\left(\frac{n-3}{2}\right)^2 + \left(\frac{j\pi}{\log (R/R_{0})}\right)^2\,.
\]

Moreover, let $C_0^{\infty}([R_0,R])$ denote the space of smooth functions on $[R_0,R]$ which vanish at the end points. Then, we can obtain a basis on this space by eigenfunctions $\{g_j\}$ of (IVP), with eigenvalue $\beta _{j}$, which are orthonormal with respect to the $L^2([R_0,R],t^{n-4}dt)$-norm. 

On the other hand (cf. \cite[Lemma 6.1.4]{Simons}), given $\Gamma \subset \s^{n-1}$ a minimal hypersurface, consider the Jacobi operator
$$
\text{(J)} \qquad  J_\Gamma f:= -\Delta_{\Gamma} f -|A_{\Gamma}|^2 f ;
$$and denote by $f_i$, $i\in \mathbb{N}$, the eigenfunctions of the above operator with eigenvalue $\lambda_{i}$. Then, we can obtain a base on the space of smooth functions on $\Gamma$, $C^{\infty}(\Gamma)$, by eigenfunctions $\{f_i\}$, with eigenvalue $\lambda _{i}$, which are orthonormal with respect to the $L^2(\Gamma)$-norm.

Thus, given  $\psi(p, t) \in C_0 ^\infty (\Sigma (R))$, where $ C_0 ^\infty (\Sigma (R))$ is the space of smooth functions which vanish on $\partial \Sigma(R)$, we obtain that $\psi$ has an unique expansion (cf. \cite[Lemma 6.1.6]{Simons}) as
\[
\psi(p,t)=\sum\limits_{i,j=1}a_{ij}f_i(p)g_j(t)\,,
\]
and, using \cite[Lemma 6.1.3]{Simons}, we obtain
\begin{equation}\label{Simons}
Q_{\delta}(R)(\psi,\psi)=\sum\limits_{i,j=1}a^2_{ij}(\lambda_i+\beta_j ).
\end{equation}


\subsection{Spectrum of $\mathbb{J}_{\delta}$ in spherical coordinates}\label{Spectrum}

Following \cite{Simons}, using separation of variables $\psi (p,r) = f(p)u(r)$, the expression of the Laplace operator $\Delta _\delta$ on $\Sigma$ in spherical coordinates given by
$$ \Delta _\delta = \frac{\partial ^2}{\partial r ^2} +\frac{N-1}{r}\frac{\partial}{\partial r} +\frac{1}{r^2}\Delta _\Gamma ,$$ and $r^2 |A_{(\Sigma ,\delta)}|^2 = |A_\Gamma|^2$; a function $\psi(p, t) \in C_0 ^\infty (\Sigma (R))$ solution to 
$$\mathbb{J}_{\delta}(\psi)=-\lambda \psi F^{\frac{4}{N-2}},$$where $\mathbb{J}_\delta$ is given by \eqref{J}, must satisfy that 
\begin{itemize}
\item $f \in C^{\infty}(\Gamma)$ belongs to the spectrum of $J_\Gamma$; with eigenvalues $\lambda_k(\Gamma)$, $k\in \mathbb{N}$.

\item $u \in C^{\infty}([R_{0} ,R])$ is a solution to $L_k(u)= -\lambda F^{\frac{4}{N-2}}u$, where $L_k$ is a family (indexed by $k\in \mathbb{N}$) of Sturm-Liouville operators defined by
\[
L_k:=\frac{d^2}{dr^2}+\frac{N-1}{r}\frac{d}{dr}+V(r)  -\frac{\lambda_{k}(\Gamma)}{r^2}, 
\]
where 
$$V(r)=\frac{N}{N-2}\left( F^{-1}\Delta_{\delta} F\right)=\frac{m(n-1)}{2r^n}\left(\frac{2r^{n-2}}{m+2r^{n-2}} \right)^{2}.$$
\end{itemize}

Hence,
\begin{equation}\label{Lk}
L_k:=\frac{d^2}{dr^2}+\frac{N-1}{r}\frac{d}{dr} + \frac{m(n-1)}{2r^n}\left(\frac{2r^{n-2}}{m+2r^{n-2}} \right)^{2}  -\frac{\lambda_{k}(\Gamma)}{r^2}.
\end{equation}

If we consider the change $v (r)=r^{\frac{N-1}{2}}u (r)$, we obtain 
\[
L_k u = r^{-\frac{N-1}{2}}\left(\frac{d^2v_k}{dr^2}+W_{k} v_k\right), 
\]
where
\begin{equation}\label{Wk}
W_{k}=\frac{m(n-1)}{2r^n}\left(\frac{2r^{n-2}}{m+2r^{n-2}} \right)^{2} -\frac{4\lambda_{k}(\Gamma)+(n-2)(n-4)}{4r^2}.
\end{equation}

\begin{lemma}\label{LemV}
The positive function 
\begin{equation}\label{v}
v(r) := \left(\frac{2r^{n-2}}{m+2r^{n-2}}\right)^{\frac{1}{n-2}}, \quad r \geq R_{0},
\end{equation}satisfies 
\begin{equation}\label{Lv}
Lv := \frac{d^2 \,v }{dr^2}+\frac{m(n-1)}{2r^n} \left(\frac{2r^{n-2}}{m+2r^{n-2}} \right)^{2} v = 0 .
\end{equation}
\end{lemma}
\begin{proof}
A straightforward computation shows that 
$$ v' (r) = \frac{m}{2r^{n-1}}\left(\frac{2r^{n-2}}{m+2r^{n-2}} \right)^{\frac{1}{n-2}+1} \text{ and } 
	v''(r) = -\frac{(n-1)m}{2r^n} \left(\frac{2r^{n-2}}{m+2r^{n-2}} \right)^{\frac{1}{n-2} +2} ,$$which shows the lemma. 
\end{proof}


\section{Proof of Theorem \ref{main1}}

Assume that $\Gamma$ is totally geodesic, hence $\lambda_1(\Gamma) = 0$. For $n\geq 4$, \eqref{Wk} implies 
$$ W_{1} \leq \frac{m(n-1)}{2r^n}\left(\frac{2r^{n-2}}{m+2r^{n-2}} \right)^{2} ,$$hence Lemma \ref{LemV}, \eqref{Lk} and \eqref{Lv} imply that $u(r)=r^{-\frac{N-1}{2}}v(r)$ satisfies 
$$ L_1 u = r^{-\frac{N-1}{2}}\left(\frac{d^2v}{dr^2}+W_{1} v \right) \leq r^{-\frac{N-1}{2}} Lv = 0 , $$where $v$ is given by \eqref{v}. Hence, the function 
$$\psi (p,r) = F^{-1} (r)u(r) = \frac{2 r^{\frac{n-2}{2}}}{m+2r^{n-2}} $$ is a positive function on $\Sigma$ such that $J_\Sigma (\psi) \leq 0 $. Also, a straightforward computation shows 
$$ \frac{\partial \psi}{\partial r}(p,r) = \frac{(n-2) r^{\frac{n}{2}-2}}{\left(m+2r^{n-2} \right) ^2} \left(m-2r^{n-2}\right) .$$

Since $2 R_0 ^{n-2} = m$, we can check that $\dfrac{\partial \psi }{\partial r} (p, R_0) = 0$. Therefore, the Fischer-Colbrie Criterion, Lemma \ref{LemFC}, implies that $\Sigma $ is stable. This proves Theorem \ref{main1}.

\subsection{Proof of Theorem \ref{main1a}}

In the case that $\Gamma$ is the Clifford torus, we know that $\lambda _{1} (\Gamma)= -(n-2)$. Hence, for $n\geq 8$, Lemma \ref{LemV}, \eqref{Lk}, \eqref{Wk} and \eqref{Lv} imply that $ L_1 u \leq 0$, $u(r)=r^{-\frac{N-1}{2}}v(r)$ where $v$ is given by \eqref{v}. Hence, using the Fischer-Colbrie Criterion we can show that $\Sigma$ is stable as above. This proves Theorem \ref{main1a}.


\subsection{Proof of Theorem \ref{newMAIN}}

In this case, the condition $4\lambda _1 (\Gamma) + (n-2)(n-4) \geq 0$ implies that $L_1 u \leq 0$, $u(r)=r^{-\frac{N-1}{2}}v(r)$ where $v$ is given by \eqref{v}. Thus, following the above ideas we can prove Theorem \ref{newMAIN}. 


\section{Proof of Theorem \ref{main2}}

Now, let $\Gamma \subset \s ^{n-1}$ be a compact minimal hypersurface in the $(n-1)-$dimensional sphere. Fix $R>R_{0}$ and consider $\psi _{j} (p,r):= f_{1} (p) g_{j} (r)$, $x= (p,r) \in \Gamma \times [R_{0},+\infty) = C_{\Gamma}$, which vanishes on $\partial \Sigma (R)$. It is well-known that the first eigenvalue of (J) satisfies $\lambda _{1} \leq -(n-2)$ and 
\[
\beta_j=\left(\frac{n-3}{2}\right)^2 + \left(\frac{j\pi}{\log (R/R_{0})}\right)^2\,,  \text{ for every }j\in \mathbb{N}.
\]

Hence, Lemma \ref{LemRelation2} and \eqref{Simons} imply
$$Q^g_R(F^{-1}\psi_{j},F^{-1}\psi_{j}) = \left(\lambda _{1} +\beta_j\right) -(n-1)m\int_{\Sigma(R)}\frac{2r^{n-4}}{\left(m+2r^{n-2}\right)^{2}} \psi_{j}^2 d\delta .
$$

Using Fubbini and the expression of $d\delta = r^{n-2} dr d\Gamma$ as a product metric, $d\Gamma$ the volume element of $\Gamma \subset \s^{n-1}$, we get 
\[
Q^g_R(F^{-1}\psi_{j},F^{-1}\psi_{j}) = \left(\lambda _{1} +\beta_j\right) + \widetilde G_{j}(R),
\]where
$$ \widetilde G_{j}(R)  = -(n-1) m c_j^2\int_{R_{0}}^{R} \frac{2r^{n-2}}{\left(m+2r^{n-2}\right)^{2}} \sin ^2\left(\frac{j\pi \log(r/R_{0})}{\log (R/R_{0})} \right) \frac{dr}{r} .$$

Consider the change of variable 
$$s= \frac{\pi}{\log (R/R_{0})}\log(r/R_{0}) \text{ and } \frac{2a(R)}{(n-2)\pi} ds =  \frac{dr}{r};$$where 
\begin{equation}\label{aR}
a(R) = \frac{(n-2)\log (R/R_{0})}{2 \pi},
\end{equation}
then 
$$ \widetilde G_{j}(R):= \frac{(n-1)j}{2\pi}\int_{0}^{\pi} \cosh ^{-2} (a(R)s) \sin^{2}(js) \, ds .$$

Hence, following the exact same computations as above we achieve 

$$ Q^g_R(F^{-1}\psi_{j},F^{-1}\psi_{j}) = \lambda _{1} + \left(\frac{n-3}{2}\right)^2 + G_{j}(R) ,$$where
$$ G_{j}(R):= \left(\frac{(n-2)j}{2a(R)}\right)^2 - \frac{(n-1)j}{2\pi}\int_{0}^{\pi} \cosh ^{-2} (a(R) j s) \sin ^{2} \left(s \right) \, ds .$$

Therefore, if $4 \leq n\leq 7$ then $\lambda _{1} + \left(\frac{n-3}{2}\right)^2 <0$ and, since $G_{j}(R) \to 0 $ as $R\to +\infty$ for all $j \geq 1$, we obtain that 
$$Q_{\Sigma}(R)(F^{-1}\psi_{j},F^{-1}\psi_{j}) <0 \text{ for all } j\geq 1 \text{ and } R \text{ large enough depending on } j,$$that is, ${\rm Ind}_{D}(\Sigma (R)) \to +\infty$ as $R\to +\infty$, since the $\psi_{j}'s$ are linearly independent. This and Claim B prove Theorem \ref{main2}.


\subsection{Proof of Theorem \ref{main3}}

This result follows from Theorem \ref{main2}.


\section{Density over minimal cones}

In this section, following ideas of \cite{SBre13,RaMo}, we represent the Schwarzschild manifold as $M^n=\mathbb{S}^{n-1} \times (s_0, +\infty)$, $n\geq3$, endowed with the metric 
\[
g_{Sch}=\frac{1}{1-2ms^{2-n}}ds^2+s^2g_{\mathbb{S}^{n-1}}\,,
\]
where $s_0=(2m)^{\frac{1}{n-2}}$. We define a continuous function $F:[s_0, +\infty)\rightarrow\mathbb{R}$ by $F'(s)=\frac{1}{\sqrt{1-2ms^{2-n}}}$ and $F(s_0)=0$. Making the change  $r=F(s)$, the metric $g_{Sch}$ can be rewritten as
\[
g_{Sch}=dr^2+h^2(r)g_{\mathbb{S}^{n-1}}\,,
\]
where $h:[0,+\infty)\rightarrow [s_0,+\infty)$ denotes the inverse of $F$. Hence, from \cite[Section 5]{SBre13}, we obtain 
\[
h'(r)=\sqrt{1-2ms^{2-n}}\,,
\]
where $s=h(r)$. The variable $r=r(x)$ represents the Schwarzschild distance to the
horizon $\partial M$. Moreover, consider the function $f (x)=h' (r(x))$ and the conformal vector field
\[
X = h(r)\partial_r,
\]
where $\partial_r$ denote the unit length radial vector, i.e., the gradient, with respect to the Schwarszchild metric $g_{Sch}$, of the function $r$. Using this explicit expression of $X$, we can obtain the divergence of $X$, on $\Sigma$, with respect to the metric $g_{Sch}$:
\[
div_{\Sigma}(X)=(n-1)f\,.
\]

In fact, following \cite[Section 2]{SBre13}, we can write $X=\nabla_{\Sigma}\varphi$, for some function $\varphi$ such that $Hess_{\Sigma}\varphi=f g$. Hence, 
\[
div_{\Sigma}X=\sum\limits_{k=1}^{n-1}Hess_{\Sigma}\varphi(E_k,E_k)-Hg(X,N)=(n-1)f\,,
\]
where $\{E_1,..., E_{n-1}\}$ is an orthonormal basis of the tangent space to $\Sigma$ and $H=0$. The function $f$ is nothing but the static potential associated to the Schwarschild manifold.

Let $\Sigma$ be a properly embedded free boundary minimal hypersurface in $(M^n,g_{Sch})$. Let $B_{\rho}$ denote the set of points in $(M^n,g_{Sch})$ at Schwarzschild distance to the horizon at most $\rho$. 

\begin{definition}
Let $\Gamma$ be a closed minimal hypersurface in the Euclidean unit sphere $\mathbb{S}^{n-1}$. We define the $\Gamma$-density at infinity of a properly embedded free boundary minimal hypersurface $\Sigma$ in the Schwarzschild manifold by
\[
\Theta_{\Gamma}(\Sigma):=\lim\limits_{\rho\rightarrow +\infty}\frac{vol(\Sigma\cap B_{\rho})}{vol(\Sigma_{\Gamma}\cap B_{\rho})}
\]
whenever this limit exists, where $\Sigma_{\Gamma}$ denotes the cone over $\Gamma$.
\end{definition}

Observe that when $n =3$ the only minimal cone is the one given by a great circle in $\mathbb{S}^2$. In our case, in higher dimension, we have a plethora of minimal cones in order to consider a density. Denote by $|S|$ the volume of a hypersurface $S$ in the Euclidean unit sphere $\mathbb{S}^{n-1}$. With this notation, we can announce the following result obtained by R. Montezuma \cite{RaMo} when $n=3$. 

\begin{theorem}
If $\Theta _\Gamma (\Sigma)$ exists and is finite; the following formula is valid:
\begin{equation}\label{ThetaArea}
\Theta_{\Gamma}(\Sigma)=\frac{area(\partial\Sigma)}{2m|\Gamma|}+\frac{n-1}{|\Gamma|}\int_{\Sigma}\frac{f}{h^{n-1}(r)}|\partial_r^{\perp}|^2_gdv\,.
\end{equation}
\end{theorem}
\begin{proof}
We follow \cite[Section 3]{RaMo} with minor changes due to the dimension. For every $0 <\sigma<\rho$, consider the vector field $W$ defined by the expression
\[
W(x) :=\left\{ \begin{matrix}
\left( \frac{1}{h^{n-1}(\sigma)}-\frac{1}{h^{n-1}(\rho)}\right)X(x)  & \text{ if } x\in M \text{ with }0 \leq r(x) \leq \sigma , \\[3mm]
\left( \frac{1}{h^{n-1}(r)}-\frac{1}{h^{n-1}(\rho)}\right)X(x) &  \text{ if } x\in M \text{ with } \sigma \leq r(x)\leq \rho , \\[3mm]
0 & \text{ otherwise.}
\end{matrix}\right.
\]

Hence, we have
\[
div_{\Sigma}W(x)=\left\{ \begin{matrix}
\left( \frac{1}{h^{n-1}(\sigma)}-\frac{1}{h^{n-1}(\rho)}\right)(n-1)f & \text{ for }0 \leq r(x) < \sigma , \\[3mm]
\frac{(n-1)f}{h^{n-1}(\rho)}+\frac{(n-1)f}{h^{n-1}(r)}|\partial_r^{\perp}|^2_g  & \text{ for } \sigma < r(x) < \rho ,
\end{matrix}\right.
\]where $\partial_r^{T}$ and $\partial_r^{\perp}$ denote the tangential and normal components of $\partial_r$, respectively, relative to the tangent spaces of $\Sigma$. Note that $div_{\Sigma}W(x)=0$ if $r(x)>\rho$.

Therefore, for almost all $0 < \sigma < \rho$, we obtain
\begin{equation*}
\begin{split}
\frac{1}{h^{n-1}(\rho)}\int_{\Sigma_{\rho}}fdv&=\frac{1}{h^{n-1}(\sigma)}\int_{\Sigma_{\sigma}}fdv +\int_{\Sigma_{\rho}\setminus\Sigma_{\sigma}}\frac{f}{h^{n-1}(r)}|\partial_r^{\perp}|^2_gdv\\
&\qquad -\frac{1}{n-1}\int_{\partial \Sigma}g(W^T,\nu)ds\,,
\end{split}
\end{equation*}
where $W^{T}$ denotes the tangential component of $W$. Here, we have used that the divergence over $\Sigma$
of the normal component of $W$ vanishes, since $\Sigma$ is a free boundary minimal hypersurface.

Note that, since $\Sigma$ is free boundary, it follows that $W$ is tangential to $\Sigma$ and $\nu=-\partial_r$. Then,
\[
g(W^T,\nu)=-\left( \frac{1}{h^{n-1}(\sigma)}-\frac{1}{h^{n-1}(\rho)}\right)h(0)=-(2m)^{\frac{1}{n-2}}\left( \frac{1}{h^{n-1}(\sigma)}-\frac{1}{h^{n-1}(\rho)}\right)\,,
\]
and, consequently, for $0 \leq \sigma < \rho$, we have

\begin{equation}\label{monotonicity}
\begin{split}
\frac{\mu(\Sigma\cap B_{\rho})}{h^{n-1}(\rho)}&=\frac{\mu(\Sigma\cap B_{\sigma})}{h^{n-1}(\sigma)}
+\int_{\Sigma_{\rho}\setminus\Sigma_{\sigma}}\frac{f}{h^{n-1}(r)}|\partial_r^{\perp}|^2_gdv\\  
&\qquad +\frac{(2m)^{\frac{1}{n-2}}}{n-1}\left( \frac{1}{h^{n-1}(\sigma)}-\frac{1}{h^{n-1}(\rho)}\right)area(\partial\Sigma)\,, 
\end{split}
\end{equation}
where $area(\partial\Sigma)$ represents the area of the boundary $\partial\Sigma$, and $\mu$ is the measure defined
by $\mu(A)=\int_Af$. Now, if $\Theta_{\Gamma}(\Sigma)$ exists, we can check that
\[
\lim\limits_{\rho\rightarrow +\infty}\frac{area(\Sigma\cap B_{\rho})}{area(C_{\Gamma}\cap B_{\rho})}=\lim\limits_{\rho\rightarrow +\infty}\frac{(n-1)area(\Sigma\cap B_{\rho})}{|\Gamma|h^{n-1}(\rho)}\,.
\]

Therefore,  assuming that $\Theta_{\Gamma}(\Sigma)$ exists, we let  $\sigma=0$ and $\rho\rightarrow+\infty$ in the identity \eqref{monotonicity} to conclude that \eqref{ThetaArea} holds.
\end{proof}

Since the integral term at the right hand side of \eqref{ThetaArea} is non-negative, we obtain: 
\begin{corollary}\label{CorDensity}
The area of the boundary of $\Sigma$ satisfies
\begin{equation}\label{CorEquality}
area(\partial \Sigma)\leq 2m|\Gamma|\Theta_{\Gamma}(\Sigma)\,.
\end{equation}
Moreover, equality holds if and only if $\Sigma$ is a minimal cone; in such case
$$ |\Gamma| \Theta_{\Gamma}(\Sigma)  = |R_0 ^{-1} \partial \Sigma| ,$$where $R_0 ^{-1}\partial\Sigma \subset \mathbb{S}^{n-1}$ is nothing but the dilation of $\partial \Sigma$ into the $(n-1)-$sphere of radius one.
\end{corollary}
\begin{proof}
The inequality follows from \eqref{ThetaArea}. In the case of equality, $\Sigma$ must be a cone and, in particular, $\partial \Sigma $ is a minimal hypersurface in $(S(R_0) ,g_{Sch})$. We can easily compute 
$$ area(\partial \Sigma) = 2m |R_0^{-1}\partial \Sigma| ,$$which finishes the proof. 
\end{proof}

At this point, we might consider different minimal cones in order to establish rigidity results in terms of the density

\subsection{Proof of Theorem \ref{ThMN}}
Since we are assuming equality in \eqref{CorEquality}, $\Sigma$ must be a cone and, in particular, $\partial \Sigma $ is a minimal hypersurface in $(S(R_0) ,g_{Sch})$ and $|R_0 ^{-1} \partial \Sigma| \leq 2 \pi ^2$. Hence, \cite{MaNe} implies that $R_0 ^{-1}\partial \Sigma \subset \mathbb{S}^3$ is either a great sphere or a Clifford torus, which proves this theorem.  

\subsection{Proof of Theorem \ref{ThAllard}}
In any dimension, by the Monotonicity Formula and Allard's Regularity Theorem \cite{Allard}, we can obtain Theorem \ref{ThAllard}.


\end{document}